\documentclass[11pt, a4paper]{article}
\usepackage[utf8]{inputenc}
\usepackage{footnote}
\usepackage{amsmath, amssymb, amsfonts, amsthm, mathtools, xcolor}
\usepackage{array}
\usepackage{systeme}
\usepackage{tikz}
\usepackage{caption}
\usepackage{subcaption}
\usepackage{hyperref}

\newcommand{\setbuilder}[2]{\left\{#1\ \colon #2\right\}}

 \def\r{1}
 \def\rr{1.3}

\newcommand\barc{} 
\def\barc(#1,#2,#3,#4){%
   \draw [teal,thick,domain=#1:#2] plot ({\r*cos(\x)}, {\r*sin(\x)});					
   \draw [black,thick,domain=#2:#3] plot ({\r*cos(\x)}, {\r*sin(\x)});				
   \draw [teal,thick,domain=#3:#4] plot ({\r*cos(\x)}, {\r*sin(\x)});					
 }
 \newcommand\barcc{} 
\def\barcc(#1,#2,#3,#4,#5,#6){%
   \draw [black,thick,domain=#1:#2] plot ({\r*cos(\x)}, {\r*sin(\x)});					
   \draw [gray,thick,domain=#2:#3] plot ({\r*cos(\x)}, {\r*sin(\x)});				
    \node at ({\rr*cos(#2)},{\rr*sin(#2)}) {#5};
    \node at ({\rr*cos(#3)},{\rr*sin(#3)}) {#6};
   \draw [black,thick,domain=#3:#4] plot ({\r*cos(\x)}, {\r*sin(\x)});					
 }
  \newcommand\cyclee{} 
\def\cyclee{%
   \draw [black,dotted,domain=0:360] plot ({\r*cos(\x)}, {\r*sin(\x)});
 }
  \newcommand\barccc{} 
\def\barccc(#1,#2,#3,#4,#5,#6){%
   \draw [black,thick,domain=#1:#2] plot ({\r*cos(\x)}, {\r*sin(\x)});					
   \draw [black,thick,domain=#2:#3] plot ({\r*cos(\x)}, {\r*sin(\x)});				
    \node at ({\rr*cos(#2)},{\rr*sin(#2)}) {#5};
    \node at ({\rr*cos(#3)},{\rr*sin(#3)}) {#6};
   \draw [black,thick,domain=#3:#4] plot ({\r*cos(\x)}, {\r*sin(\x)});					
 }	
\newcommand\conn{} 
\def\conn(#1,#2){%
   \draw [teal,thick] ({\r*cos(#1)}, {\r*sin(#1)}) -- ({\r*cos(#2)}, {\r*sin(#2)});		
 }
 \newcommand\connn{} 
\def\connn(#1,#2){%
   \draw [black,thick] ({\r*cos(#1)}, {\r*sin(#1)}) -- ({\r*cos(#2)}, {\r*sin(#2)});		
 }

\newtheorem{theorem}{Theorem}[section]

\newtheorem{lemma}[theorem]{Lemma}

\newtheorem{proposition}[theorem]{Proposition}
\newtheorem{conjecture}[theorem]{Conjecture}
\theoremstyle{definition}
\newtheorem{definition}[theorem]{Definition}
\newtheorem{remark}[theorem]{Remark}
\newtheorem{assumption}{Assumption}

\title{Low independence number and Hamiltonicity implies pancyclicity}
\author{Attila Dankovics \\ \href{mailto:a.j.dankovics@lse.ac.uk}{a.j.dankovics@lse.ac.uk}}

\begin{document}
\date{}
\maketitle

\begin{abstract}
A graph on $n$ vertices is called pancyclic if it contains a cycle of every length $3\le l \le n$. Given a Hamiltonian graph $G$ with independence number at most $k$ we are looking for the minimum number of vertices $f(k)$ that guarantees that $G$ is pancyclic. The problem of finding $f(k)$ was raised by Erdős in 1972 who showed that $f(k)\le 4k^4$, and conjectured that $f(k)=\Theta(k^2)$. Improving on a result of Lee and Sudakov we show that $f(k)=O(k^{11/5})$.
\end{abstract}

\section{Introduction}

A \emph{Hamilton cycle} of a graph is a cycle that passes through all vertices.  It is difficult to decide whether a graph contains a Hamilton cycle, therefore it is valuable to establish useful sufficient conditions for Hamiltonicity. The most well known sufficient condition is by Dirac\cite{dirac}, who showed that if each vertex of an $n$ vertex graph has at least $n/2$ neighbors, then the graph is Hamiltonian. A graph is \emph{pancyclic} if it contains a cycle of every length $3\le \ell\le n$, where $n$ denotes the number of vertices. By definition pancyclicity implies Hamiltonicity. Although the converse is not true, it is often the case that conditions that imply Hamiltonicity turn out to also imply pancyclicity. A famous meta conjecture of Bondy \cite{bondyMeta} states that almost all non-trivial sufficient conditions of Hamiltonicity also implies pancyclicity with the possible exception of a few graphs.

The independence number of a graph $G$ is the size of the largest stable  set, denoted by $\alpha(G)$. A famous result of Chvátal and Erdős \cite{chvatalErdos} states that if $\kappa(G)\ge\alpha(G)$, where $\kappa(G)$ is the vertex connectivity of $G$, then $G$ is Hamiltonian. Keevash and Sudakov \cite{KeevashSudakov} showed that the similar but stronger condition $\kappa(G)\ge c\alpha(G)$, where $c>1$ is a constant, is sufficient to conclude pancyclicity. 

In this paper we study a connection between Hamiltonicity, pancyclicity and independence number. Assuming $G$ is a Hamiltonian graph  with independence number at most $k$ we are looking for the minimum number of vertices $f(k)$ that guarantees that $G$ is pancyclic. The problem of finding $f(k)$ was raised by Erdős who showed that $f(k)\le 4k^4$ and conjectured that a stronger statement holds.

\begin{conjecture} [Erdős \cite{erdosConjecture}] There are constants $c_1$ and $c_2$ such that for all $k$ we have 
$c_1k^2\le f(k)\le c_2 k^2$. 
\end{conjecture}

The following known construction due to Erdős provides a lower bound for this conjecture. Let $K_1,\dots,K_{k}$ be cliques of size $k-2$ and add an edge between successive cliques (including between the last and the first) such that these edges are independent. It is easy to check that this graph is Hamiltonian, has independence number $k$ and $k(k-2)$ vertices but does not contain a cycle of length $k-1$.

The result of Erdős was later improved by Keevash and Sudakov \cite{KeevashSudakov}, who showed that $f(k)\le 150k^3$ holds and by Lee and Sudakov \cite{LeeSudakov}, who proved that $f(k)=O(k^{7/3})$ holds.

Here we improve their results.

\begin{definition}
For $\beta>0$ let $SC1(\beta)$ denote the following statement. There exists $c>0$ such that given a Hamiltonian graph $G$ with $n\ge ck^{11/5}$ vertices and independence number at most $k$ and a subset of vertices $W$ with at most $20k^2$ vertices, we can find a cycle of length $n-1$ containing all the vertices from $W$.
\end{definition}

\begin{theorem} [Lee, Sudakov \cite{LeeSudakov}] \label{lsequiv}
For all $\beta\ge 2$, assuming $SC1(\beta)$ the following statement holds. There exists $c'>0$, such that if $G$ is a Hamiltonian graph with $n\ge c'k^{11/5}$ vertices and independence number at most $k$, then $G$ is pancyclic.
\end{theorem}

The above theorem is implicitly proved in \cite{LeeSudakov}. To see this one follows their \emph{Proof of Theorem 1.1}, this gives the stronger conclusion if their Theorem 2.1 is replaced by  $SC1(\beta)$.

\begin{theorem} \label{mainEquivalent}
	$SC1(11/5)$ holds.
\end{theorem}

The goal of the following sections is to prove Theorem \ref{mainEquivalent}. Using Theorem \ref{lsequiv} and Theorem \ref{mainEquivalent} we get the following immediate corollary. 

\begin{theorem} \label{mainTheorem}
There exists $c>0$, such that if $G$ is a Hamiltonian graph with $n\ge ck^{11/5}$ vertices and independence number at most $k$, then $G$ is pancyclic.
\end{theorem}

The goal of the following sections is to prove Theorem \ref{mainEquivalent}.

For a proof of Theorem \ref{mainEquivalent} we substantially extend the methods of \cite{LeeSudakov}. The improvement comes from Lemma \ref{lemma41} and the inductive approach to proving Lemma \ref{mainInduction} which is made possible by Lemma \ref{lemma42}. Proving these new lemmas constitutes most of Section \ref{mainSection}. Before that, we will state some definitions and prove a basic structural proposition in Section

\section{Definitions, earlier results} \label{sec:}

The goal of this section is to state the basic definitions and to prove Proposition \ref{simpleArc} which states the existence of a structure we will use in Section \ref{mainSection}.

We make no attempt to find the optimal value of $c$ in Theorem \ref{mainTheorem}. For this reason we can ignore small rounding errors and thus will omit all floor and ceiling signs. We fix a large constant $c$, how large we actually need will come from later calculations so we do not specify at this point.

\begin{assumption} \label{assumptions}
From this point we assume for a contradiction to Theorem \ref{mainEquivalent} the following;
\begin{itemize}
\item $G$ is a Hamiltonian graph with $n\ge ck^{11/5}$ vertices,
\item $G$ has independence number at most $k$,
\item $W$ is a subset of $V(G)$ with at most $20k^2$ vertices,
\item $G$ has no cycle of length $n-1$ containing $W$,
\item $H$ is a Hamilton cycle in $G$.
\end{itemize}
\end{assumption}

We call a vertex of $G$ \emph{problematic} if it has degree at most $2k$ or is an element of $W$. There are at most $2k^2$ vertices with degree at most $2k$ (by the greedy algorithm for finding independent sets), so there are at most $22k^2$ problematic vertices. The motivation for calling these vertices problematic comes from Proposition \ref{contradictingcycle}.

We call a cycle $C$ a \emph{contradicting cycle} if it has length $n-k\le |C|\le n-1$ and contains all problematic vertices.

The following proposition shows that in the graphs we are considering no contradicting cycle exists, thus justifying their name.

\begin{proposition} [\cite{LeeSudakov} Proposition 3.1] \label{contradictingcycle}
	If $G$ satisfies Assumption \ref{assumptions}, then there is no contradicting cycle in $G$.
\end{proposition}




We say two vertices of $G$ are \emph{consecutive} if they are neighbors in $H$. A set of vertices is \emph{continuous} if they form a path in $H$. For a subset $A\subset V(G)$ the \emph{continuous closure} of the set is $\overline{A}$ the minimum sized continuous set that contains it (in general this might not always be unique, but we will use it only in cases when it is).

Next we define arc-systems, which are the objects that we will primarily use in the rest of the paper. 

\begin{definition}
A family of subsets of $V(G)$ (where the graph $G$ has a fixed Hamilton cycle $H$) denoted by $\mathcal{A}$ is called an \emph{arc-system} and its elements \emph{arcs} if the following hold.
\begin{itemize}
    \item For all $A\neq B \in \mathcal{A}$, we have $\overline{A}\cap\overline{B}=\emptyset$, that is, the continuous closure of arcs are pairwise disjoint.
    \item For all $A\in\mathcal{A}$, we have $\overline{A}\cap W=\emptyset$, that is, the continuous closure of each arc has no problematic vertex in it.
    \item For all $A\in\mathcal{A}$, no two vertices in $A$ are consecutive.
\end{itemize}
\end{definition}

\begin{remark} \label{remIndep}
If $A$ is an arc and $|\overline{A}|\le k+2$ then $A$ is an independent set. Indeed, if there was an edge $\{u,v\}$ where $u,v\in A$ then using the longer path between $u$ and $v$ in $H$ and the edge $\{u,v\}$ we would get a contradicting cycle (see Figure \ref{fig:onearc}).
\end{remark}

\begin{definition}
We call an arc system $\mathcal{A}$ \emph{independent} if it has the property that for all $A$ in $\mathcal{A}$ we have $|\overline{A}|\le k$.
\end{definition}

We say the \emph{size} of the arc-system is $|\mathcal{A}|$ and the \emph{length} of the arc system is $\min_{A\in\mathcal{A}} |A|$.

\begin{proposition}
Given $c_1$ and $c_2$, there exists $c$ such that if we assume Assumption  \ref{assumptions} then there is an independent arc-system in the graph $G$ of size $c_1k^2$ and length $c_2k^{1/5}$.
\end{proposition}

\begin{proof}
We start with the empty arc-system. Removing the problematic vertices from $H$, we obtain a set $\mathcal{P}$ of at most $22k^2$ paths. From this set we will construct an arc-system $\mathcal{A}$ with the desired properties.

While there is a path $\{v_1,v_2,\dots,v_m\}=P\in\mathcal{P}$ such that $m\ge 2c_2k^{1/5}$ we do the following. Remove $P$ from $\mathcal{P}$. Add $\{v_{2c_2k^{1/5}+1},v_{2c_2k^{1/5}+2},\dots,v_m\}$ to $\mathcal{P}$. Add $\{v_1,v_3,\dots,v_{2c_2k^{1/5}-1}\}$ to $\mathcal{A}$. In words, we remove the first $2c_2k^{1/5}$ vertices of $P$ and form an arc from every second vertex in it, and add that arc to 
$\mathcal{A}$.

At the end of this process we have at most $(22k^2)(2c_2k^{1/5})$ leftover vertices (from paths shorter than $2c_2k^{1/5}$), we removed $22k^2$ problematic vertices at the start, and the half of the other vertices were used to form arcs in $\mathcal{A}$. So we have at least $\frac{ck^{11/5}-(22k^2)(2c_2k^{1/5}+1)}{2c_2k^{1/5}}$ arcs in $\mathcal{A}$, which is more than $c_1k^2$ if $c$ is large enough.
\end{proof}

Fix one such arc-system $\mathcal{A}$. Next we get rid of matchings of size 2 between arcs.

\begin{definition}
We say a graph is \emph{$M_2$-free} if it doesn't have two independent edges (or equivalently, there is a vertex that is incident to all edges).
\end{definition}

\begin{definition}
We say an arc-system is \emph{simple} if it is independent and for each pair of arcs, the subgraph of $G$ induced by them is $M_2$-free. 
\end{definition}

We draw the vertices of the graph $G$ on a circle in the plane in the order of the cycle $H$ and connect neighboring vertices with line segments. If there are two independent edges $\{x_1,y_1\}$, $\{x_2,y_2\}$ between arcs $A_1$ and $A_2$ then these can be intersecting on this drawing or not. 

If they are intersecting then we immediately find a contradicting cycle the following way. From $H$ remove the edges of the shorter path between $x_1,x_2$ and, similarly, remove the edges of the shorter path between $y_1,y_2$ and instead add the edges $\{x_1,y_1\}$ and $\{x_2,y_2\}$ (see Figure \ref{fig:intersectingM2}). This gives a cycle with at least $n-2c_2k^{1/5}>n-k$ vertices. By Proposition \ref{contradictingcycle} this is a contradiction to Assumption \ref{assumptions}.

\captionsetup[subfigure]{format=hang}
\begin{figure} 
  \begin{subfigure}{.3\textwidth}
    \begin{tikzpicture}[scale=1]
        \cyclee
        \barcc(220,240,300,320,$u$,$v$)
        \connn(240,300)
    \end{tikzpicture}
    \centering
    \caption{An arc with an edge inside}
    \label{fig:onearc}
  \end{subfigure}
  \quad
  \begin{subfigure}{.3\textwidth}
    \begin{tikzpicture}[scale=1]
        \cyclee	
        \barcc(235,250,290,305,$x_1$,$x_2$)
        \barcc(10,25,65,80,$y_1$,$y_2$)
        \connn(25,250)
        \connn(65,290)
	\end{tikzpicture}
    \centering
    \caption{An intersecting\\$M_2$ between arcs}
    \label{fig:intersectingM2}
  \end{subfigure}
  \quad
  \begin{subfigure}{.3\textwidth}
    \begin{tikzpicture}[scale=1]
        \cyclee
        \barcc(250,260,280,290,$x_1$,$x_2$)
        \barcc(40,50,70,80,$y_2$,$y_1$)
        \barcc(150,160,180,190,$y_4$,$y_3$)
        \barcc(330,340,360,370,$x_3$,$x_4$)
        \connn(70,260)
        \connn(50,280)
        \connn(160,360)
        \connn(180,340)
    \end{tikzpicture}
    \centering
    \caption{Two non-intersecting $M_2$ between arcs}
    \label{fig:doubleM2}
  \end{subfigure}
  \centering
  \caption{Contradicting cycles implied by edge configurations}
\end{figure}
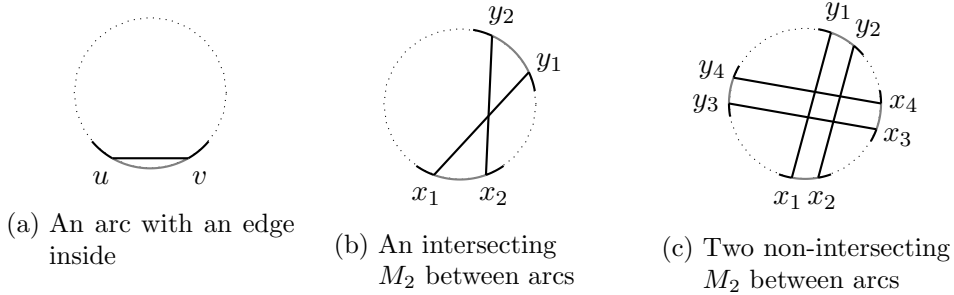

If we find two pairs of arcs each with non-intersecting $M_2$ such that these two $M_2$ ($\{x_1,y_1\}$, $\{x_2,y_2\}$ between $A_1$ and $A_2$ and $\{x_3,y_3\}$, $\{x_4,y_4\}$ between $A_3$ and $A_4$) intersect each other on the drawing, then again we can find a contradicting cycle the following way. From $H$ remove the edges of the shorter path between $x_1,x_2$; $y_1,y_2$; $x_3,x_4$; $y_3,y_4$ and instead add the edges $\{x_1,y_1\}$, $\{x_2,y_2\}$, $\{x_3,y_3\}$, $\{x_4,y_4\}$ (see Figure \ref{fig:doubleM2}). This gives a cycle with at least $n-4c_2k^{1/5}>n-k$ vertices. By Proposition \ref{contradictingcycle} this is a contradiction to Assumption \ref{assumptions}.

This implies that if we look at the graph where the vertices are the arcs and the edges are $M_2$ subgraphs between them, then this is a planar graph. That implies 5 colourability. By taking the majority colour we get a simple arc-system with size $\frac{c_1}{5}k^2$ and length $c_2k^{1/5}$.

This leads to the following proposition:

\begin{proposition} \label{simpleArc}
Given $c_1$ and $c_2$, there exists $c$ such that if we assume Assumption \ref{assumptions} then in the graph $G$ there is a simple arc-system of size $c_1k^2$ and length $c_2k^{1/5}$. \qed
\end{proposition}

In some cases we want to consider an auxiliary graph, in which the arcs are the vertices.

\begin{definition}
The \emph{arc-graph} of an arc-system $\mathcal{A}$ is the graph $G_{\mathcal{A}}$ where the vertices are the arcs in the system and there is an edge between two arcs in the arc-graph if and only if there is an edge between the two arcs in $G$. We call the edges of the arc-graph \emph{arc-edges}.
\end{definition}

\section{Key lemmas, proof of the main theorem}
\label{mainSection}

The goal of this section is to prove Lemma \ref{mainInduction} which will imply Theorem \ref{mainTheorem}. To prove it, we will use induction and two structural lemmas \ref{lemma41}, \ref{lemma42}.

Given an arc-system $\mathcal{A}$, let $G[\mathcal{A}]$ denote the subgraph induced by all vertices in the arcs of $\mathcal{A}$.

\begin{lemma} \label{lemma34}
Given a simple arc-system $\mathcal{A}$ with length $a$ and size $b$ and $m$ arc-edges in the corresponding arc-graph, there is an independent set in $G[\mathcal{A}]$ of size $ab-m$.
\end{lemma}

\begin{proof}
Since the arc-system is simple, the edges of $G[\mathcal{A}]$ corresponding to a single arc-edge $e$ can be covered by a single vertex of $G[\mathcal{A}]$ (that is, the is a vertex $v$ in $G[\mathcal{A}]$ such that all the edges corresponding to $e$ are incident to $v$). Removing these vertices we get an independent vertex set in $G[\mathcal{A}]$ of size $ab-m$.
\end{proof}

Now we define the function that we want to work with.

\begin{definition}
Let $g(a,b)$ denote the largest number such that given a simple arc-system $\mathcal{A}$ of length $a$ and size $b$ there is always an independent set of size $g(a,b)$ in $G[\mathcal{A}]$.
\end{definition}

\begin{remark} \label{apbp}
For all $p\in \mathbb{N}$ we define the following constants that we will use in the following section;
$a_p=10\cdot 3^p$, $b_p=1000^p\cdot 4^{p^2}$.
\end{remark}

The lemma that we want to prove in this section is the following.

\begin{lemma} \label{mainInduction}
Using the constants from Remark \ref{apbp}, for every $p\in\mathbb{N}$ and for every $x$ with $a_px\le k$ we have $$g\left(a_px,b_px^{p(p-1)/2}\right)\ge x^p+1\,.$$
\end{lemma}

This means that given $\Theta\left(x^{p(p-1)/2}\right)$ arcs with each having at least $\Theta(x)$ vertices we can find an independent set of size $\Theta(x^p)$.

First we prove Theorem \ref{mainTheorem} using Lemma \ref{mainInduction}.

\begin{proof} [Proof of Theorem \ref{mainTheorem}]
Using Lemma \ref{mainInduction} for $p=5$ and $x=k^{1/5}$ and Proposition \ref{simpleArc} we arrive to an independent set of size $k+1$ and therefore a contradiction which proves Theorem \ref{mainEquivalent} and thus Theorem \ref{mainTheorem}.
\end{proof}

The following remark provides the base case of the inductive proof of Lemma \ref{mainInduction}.

\begin{remark} \label{mainIndRemark}
To prove Lemma \ref{mainInduction} we will use induction. As the base case we observe that $g(x+1,1)=x+1$ (by Remark \ref{remIndep}). Also $g(2x,2x)\ge x^2+1$, since an arc-system of size $2x$ can have at most $\frac{2x(2x-1)}{2}$ arc-edges, therefore by Lemma \ref{lemma34} it has an independent set of size $4x^2-\frac{2x(2x+1)}{2}\ge x^2+1$.
\end{remark}

To prepare for the induction step, first we prove Lemmas \ref{lemma41} and \ref{lemma42}.

\begin{definition}
Given an arc-system $\mathcal{A}$ of the graph $G$ and a subset of an arc $X\subset A\in\mathcal{A}$, we say that the \emph{arc-neighborhood} of $X$ is $$N_{\mathcal{A}}(X)=\setbuilder{B}{(\exists b\in B) (\exists x\in X) \{b,x\}\in E(G)}.$$ We sometimes write $N_\mathcal{A}(v)$ meaning $N_\mathcal{A}(\{v\})$ for simplicity. We denote by $d_\mathcal{A}(X)$ the size of the arc-neighborhood, that is $d_\mathcal{A}(X)=|N_\mathcal{A}(X)|$.
\end{definition}

\begin{definition}
Using the constants from Remark \ref{apbp} and given $p>1$ integer, we say an arc $A$ is \emph{good} in the arc-system $\mathcal{A}$, if to at least half of the vertices $v\in A$ we can assign a set of $$4b_{p-1}x^{(p-1)(p-2)/2}$$ arcs in $N_\mathcal{A}(v)$, each arc of $\mathcal{A}$ being assigned to at most one $v\in\mathcal{A}$. If an arc is not good we call it \emph{bad}.
\end{definition}

Also we say that a subset $X$ of an arc $A\in\mathcal{A}$ is \emph{expanding} in $\mathcal{A}$, if $(\forall Y\subset X) d_\mathcal{A}(Y)\ge |Y|4b_{p-1}x^{(p-1)(p-2)/2}$.

The definition of good and expanding depends on $p$, but it will always be clear from the context which $p$ is meant.

We will use the following simple proposition in the proof Lemma \ref{lemma41}.

\begin{proposition} \label{submod}
	Given an arc-system $\mathcal{A}$ of a graph $G$, the function $d_\mathcal{A}$ is submodular.
\end{proposition}

\begin{proof}
We trivially have 
$$|N_\mathcal{A}(A)|+|N_\mathcal{A}(B)|= |N_\mathcal{A}(A)\cup N_\mathcal{A}(B)|+|N_\mathcal{A}(A) \cap N_\mathcal{A}(B)|$$
and $|N_\mathcal{A}(A\cap B)|\le|N_\mathcal{A}(A) \cap N_\mathcal{A}(B)|$ as the former is a subset of the latter. Furthermore, $N_\mathcal{A}(A)\cup N_\mathcal{A}(B)=N_\mathcal{A}(A \cup B)$. Putting these together we get 
$$d_\mathcal{A}(A)+d_\mathcal{A}(B)\ge d_\mathcal{A}(A\cup B)+d_\mathcal{A}(A\cap B)\,.$$
\end{proof}

\begin{lemma} \label{lemma41}
Using the constants from Remark \ref{apbp}, for each integer $p>1$, given a simple arc-system $\mathcal{A}$ of size $b_px^{p(p-1)/2}$ and length $a_px$, there is either an independent set of size $x^p+1$ in $G[\mathcal{A}]$ or there is a non-empty $\mathcal{A}'\subseteq\mathcal{A}$ such that for all $A\in\mathcal{A}'$, $A$ is good in $\mathcal{A}'$.
\end{lemma}

\begin{proof}
We observe that an arc $A$ is good if and only if it has an expanding subset of size at least $|A|/2$, by Hall's theorem.

We define a process, by the end of which we either have the required independent set or the good subset. Let $\mathcal{A}_0=\mathcal{A}$ be the given arc system of size $b_px^{p(p-1)/2}$ and length $a_px$. Let $\mathcal{B}_0$ be the empty system. In step $t$ we will define arc-systems $\mathcal{A}_t$ and $\mathcal{B}_t$ as follows. 

If each $A\in\mathcal{A}_{t-1}$ is good in $\mathcal{A}_{t-1}$ then we define $\mathcal{A}'=\mathcal{A}_{t-1}$, $\mathcal{B}=\mathcal{B}_{t-1}$ and the process terminates.

Otherwise we take a bad arc $A$ from $\mathcal{A}_{t-1}$ and a maximal expanding set $X$ in it. We say that $Y\subset X$ is \emph{tight} if $d_{\mathcal{A}_{t-1}}(Y)< (|Y|+1)4b_{p-1}x^{(p-1)(p-2)/2}$. Let $B$ denote $A\setminus X$. Now for every $v\in B$ there is a tight set $T_v$ such that $d_{\mathcal{A}_{t-1}}(T_v \cup \{v\})< (|T_v|+1)4b_{p-1}x^{(p-1)(p-2)/2}$ (by the maximality of $X$). Let $T$ denote $\cup_{v\in B} T_v$. We claim that 
\begin{equation} \label{NB}
d_{\mathcal{A}_{t-1}}(B)\le 2a_px4b_{p-1}x^{(p-1)(p-2)/2}\,.
\end{equation} 
First we see that for each $v\in B$ we have $$|N_\mathcal{A}(v)\setminus N_{\mathcal{A}_{t-1}}(T_v)|\le 4b_{p-1}x^{(p-1)(p-2)/2}\,,$$ since $T_v$ was expanding and $T_v\cup\{v\}$ is not. This implies 
\begin{equation} \label{NBNT}
|N_{\mathcal{A}_{t-1}}(B)\setminus N_{\mathcal{A}_{t-1}}(T)|\le a_px4b_{p-1}x^{(p-1)(p-2)/2}\,.
\end{equation} 
Next we claim that if a set $T_i$ is the union of $i$ tight sets, then $$d_{\mathcal{A}_{t-1}}(T_i)\le (|T_i|+i)4b_{p-1}x^{(p-1)(p-2)/2}\,,$$ which we can prove by induction. We can assume $T_i=T_{i-1}\cup T'$ where $T_{i-1}$ is the union of $i-1$ tight sets and $T'$ is tight. Then $$d_{\mathcal{A}_{t-1}}(T_i)\le d_{\mathcal{A}_{t-1}}(T_{i-1})+d_{\mathcal{A}_{t-1}}(T')-d_{\mathcal{A}_{t-1}}(T_{i-1}\cap T')$$ $$\le ((|T_{i-1}|+i-1)+(|T'|+1)-|T_{i-1}\cap T'|)4b_{p-1}x^{(p-1)(p-2)/2}$$ where the first inequality is Proposition \ref{submod}. The second inequality follows by induction on $T_{i-1}$, tightness of $T'$ and expansion of $T'\cap T_{i-1}$.

Using that $|T|\le \frac{a_p}{2}x4b_{p-1}x^{(p-1)(p-2)/2}$ and that therefore $T$ can be written as the union of $\frac{a_p}{2}x4b_{p-1}x^{(p-1)(p-2)/2}$ tight sets, we have
\begin{equation} \label{NT}
d_\mathcal{A}(T)\le a_px4b_{p-1}x^{(p-1)(p-2)/2}\,.
\end{equation}
\eqref{NBNT} and \eqref{NT} together imply \eqref{NB}. So in this step we define $\mathcal{B}_t=\mathcal{B}_{t-1}\cup \{B\}$ and $\mathcal{A}_t=\mathcal{A}_{t-1}\setminus \{A\}$. Note that the total size of $\mathcal{A}_t$ and $\mathcal{B}_t$ is $b_px^{p(p-1)/2}$.

If by the end of this process we  have a non-empty good arc-system $\mathcal{A}'$ then we have found what we are looking for.

If $\mathcal{A}'$ is empty then we have an arc-system $\mathcal{B}$ with length $\frac{a_p}{2}x$ and size $b_px^{p(p-1)/2}$. We count the arc-edges of $G_\mathcal{B}$ in the following way. We assign each arc-edge to the arc that was added to $\mathcal{B}$ in the earlier step. By the Equation \eqref{NB} property of $B$ proven in the process, each arc will be assigned at most $2a_px4b_{p-1}x^{(p-1)(p-2)/2}$ arc-edges this way. Thus $G_\mathcal{B}$ has at most $$b_px^{p(p-1)/2}2a_px4b_{p-1}x^{(p-1)(p-2)/2}$$ arc-edges. This means $G_\mathcal{B}$ has an edge density of at most $$\frac{b_px^{p(p-1)/2}2a_px4b_{p-1}x^{(p-1)(p-2)/2}}{\binom{b_px^{p(p-1)/2}}{2}}\,.$$

We take a subset $\mathcal{C}$ of $\mathcal{B}$ of size $x^{p-1}$ with minimal amount of arc-edges. $G_\mathcal{C}$ will have at most the same edge density as $G_\mathcal{C}$. Therefore $G_\mathcal{C}$ has at most $$\binom{x^{p-1}}{2}\frac{b_px^{p(p-1)/2}2a_px4b_{p-1}x^{(p-1)(p-2)/2}}{\binom{b_px^{p(p-1)/2}}{2}}\le \frac{dx^p}{b_p}$$ edges, where $d$ is a constant not depending on $x$ or $b_p$. Therefore by Lemma \ref{lemma34} $G[\mathcal{C}]$ has an independent set of size $\frac{a_p}{2}x^p-\frac{dx^p}{b_p}\ge x^p+1$ as $a_p\ge 5$ and $b_p\ge d$.
\end{proof}

\begin{definition}
We fix a direction on the Hamilton cycle $H$, so we can order the vertices of an arc $u<v$ if $u$ is before $v$ in the given direction. We say that 3 arcs $(A,B,C)$ form a \emph{semi-triangle} if they are in the given order and there exists $a_1<a_2 \in A, b_1<b_2 \in B, c_1<c_2 \in C$ such that one of the following condition holds (see Figure \ref{fig:semi}).
\begin{itemize}
    \item Type 1: $\{a_1,c_1\},\{a_2,b_1\},\{b_2,c_2\}\in E(G)$ and $A$ and $B$ are not consecutive arcs.
    \item Type 2: $\{a_1,b_1\},\{a_2,c_1\},\{b_2,c_2\}\in E(G)$.
\end{itemize}
\end{definition}

\begin{figure}
  \begin{subfigure}{.4\textwidth}
  	\begin{tikzpicture}[scale=1]
        \cyclee
        \barccc(220,230,230,235, , )
        \barccc(250,260,280,290,$a_2$,$a_1$)
        \barccc(40,50,70,80,$c_2$,$c_1$)
        \barccc(150,160,180,190,$b_2$,$b_1$)
        \connn(260,180)
        \connn(50,160)
        \connn(70,280)
    \end{tikzpicture}
    \centering
    \caption{Type 1}
    \label{fig:semi1}
  \end{subfigure}
  \quad
  \begin{subfigure}{.4\textwidth}
  	\begin{tikzpicture}[scale=1]
        \cyclee
        \barccc(250,260,280,290,$a_2$,$a_1$)
        \barccc(40,50,70,80,$c_2$,$c_1$)
        \barccc(150,160,180,190,$b_2$,$b_1$)
        \connn(260,70)
        \connn(50,160)
        \connn(180,280)
    \end{tikzpicture}
    \centering
    \caption{Type 2}
    \label{fig:semi2}
  \end{subfigure}
  \centering
  \caption{Semi-triangles}
  \label{fig:semi}
\end{figure}
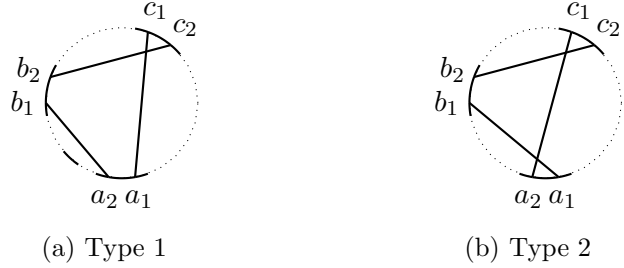

Note that a Type 2 semi-triangle gives us a contradicting cycle (see Figure \ref{fig:semicont1}), therefore it can not exist. Using the good arc-system given by Lemma \ref{lemma41} we will show the existence of certain Type 1 semi-triangles. Later, using those semi-triangles we find a contradicting cycle.

Given an arc $A$ we get the \emph{main part} of it by taking the second half of it in the given order. That is, $M(A)\subset A$, $|M(A)|=|A|/2$ and for all $v\in M(A), u\in A \setminus M(A)$ we have $v>u$. We define \emph{leftover part} as $L(A)=A\setminus M(A)$. For an arc-systems $M(\mathcal{A})$ and $L(\mathcal{A})$ are the set of the main and leftover parts respectively of each arc in the system.

\begin{lemma} \label{lemma42}
Using the constants from Remark \ref{apbp}, for each $p>1$ assuming the statement of Lemma \ref{mainInduction} for $p-1$ the following holds. Given a simple arc-system $\mathcal{A}$ of length $a_px$ and $A\in\mathcal{A}$ and to at least $|A|/6$ vertices of $A$ assigned at least $b_{p-1}x^{(p-1)(p-2)/2}+1$ neighboring main arcs, meaning that for each arc $v$ the assigned arcs are from $N_{M(\mathcal{A})}(v)$, such that each arc is assigned at most once. Then either there are arcs $B,C\in\mathcal{A}$ such that $(A,B,C)$ is a semi-triangle of Type 1, or there is an independent set in $G[\mathcal{A}]$ of size $x^p+1$.
\end{lemma}

\begin{proof}
If any of the assigned neighboring arcs is consecutively after $A$ we unassign it. Now let $v\in A$ be a vertex that was assigned neighboring arcs. We look at the corresponding leftover arcs of these neighbors, that is an arc-system $\mathcal{A}_v$ of size $b_{p-1}x^{(p-1)(p-2)/2}$ and length $a_px/2$. By Lemma \ref{mainInduction} for $p-1$ and by $a_p\ge 2a_{p-1}$ there is an independent set  $J_v$ in $G[\mathcal{A}_v]$ of size $x^{p-1}+1$. Taking $x$ of these sets, which we can do because $a_px/6>x$, we either get an independent set of size more than $x^p+1$ or an edge $\{b,c\}$ between $J_v$ and $J_u$. Let $B$ and $C$ be the arc of $b$ and $c$ respectively. Without loss of generality we might assume that $A,B,C$ are in this order on the Hamilton cycle. Then $(A,B,C)$ is a semi-triangle, because of the edge proving that $M(B)$ is a neighbor of $v$, the edge proving that $M(C)$ is a neighbor of $u$ and the edge $\{b,c\}$ going between $L(B)$ and $L(C)$. Since Type 2 semi-triangles cannot exist, this must be a Type 1 semi-triangle.
\end{proof}

With this we are ready to prove the main lemma of the section.

\begin{proof} [Proof of Lemma \ref{mainInduction}]
We have already seen that the lemma holds for $p\le 2$ in Remark \ref{mainIndRemark}. For $p\ge 3$ we use induction.

Let $\mathcal{A}'$ be a simple arc system provided by Lemma \ref{simpleArc}. Using Lemma \ref{lemma41} we obtain a subsystem $\mathcal{A}$, such that $M(\mathcal{A})$ is good; or we find an independent set of size $x^p+1$ and we are done. Using Lemma \ref{lemma42} we get that there is a Type 1 semi-triangle; or we find the desired independent set. We define the length of a semi-triangle $(A,B,C)$ as the number of arcs between $A$ and $B$ (note that by definition this is at least 1). We take the lowest-length semi-triangle of Type 1 $(A,B,C)$ in $\mathcal{A}$. Let $D$ denote the arc consecutively after $A$ and $\mathcal{A}_A$ those arcs between $A,B$; $\mathcal{A}_B$ those arcs between $B,C$; and $\mathcal{A}_C$ those arcs between $C,A$. Now by applying the pigeonhole principle twice, for one of these sub arc-systems it is true that at least $|D|/6$ vertices of $D$ are assigned at least $b_{p-1}x^{(p-1)(p-2)/2}+1$ neighboring arcs from there. If this is $\mathcal{A}_A$, then by applying Lemma \ref{lemma42} on $A$ and $\mathcal{A}_A$ we find a shorter Type 1 semi-triangle contradicting the minimality of $(A,B,C)$; or the desired independent set. If this is $\mathcal{A}_B$ or $\mathcal{A}_C$ then we get a new Type 1 semi-triangle, which together with $(A,B,C)$ implies a contradicting cycle (see Figures \ref{fig:semicont2a} and \ref{fig:semicont2b}), which is impossible. So at one of these steps we must have found the independent set we are looking for.
\end{proof}

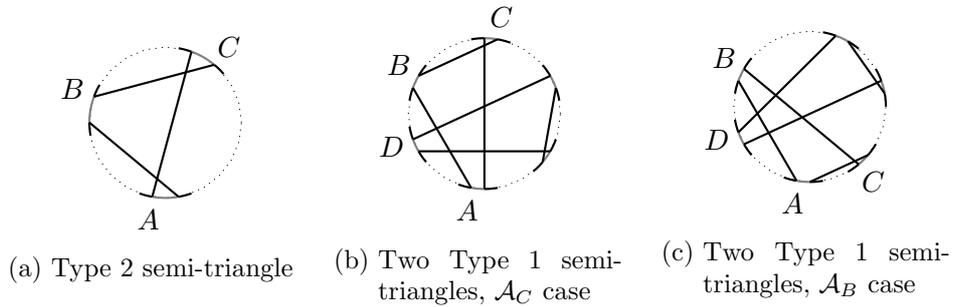
\begin{figure}
  \begin{subfigure}{.3\textwidth}
    \begin{tikzpicture}[scale=1]
        \cyclee
        \barcc(250,260,280,290,$A$,)
        \barcc(40,50,70,80,$C$,)
        \barcc(150,160,180,190,$B$,)
        \connn(260,70)
        \connn(50,160)
        \connn(180,280)
    \end{tikzpicture}
    \centering
    \caption{Type 2 semi-triangle}
    \label{fig:semicont1}
  \end{subfigure}
  \quad
  \begin{subfigure}{.3\textwidth}
    \begin{tikzpicture}[scale=1]
        \cyclee
        \barcc(250,260,270,280,$A$,)
        \barcc(70,80,90,100,$C$,)
        \barcc(140,150,160,170,$B$,)
        \connn(270,90)
        \connn(80,150)
        \connn(160,260)
        \barcc(190,200,210,220,$D$,)
        \barcc(310,320,330,340,,)
        \barcc(10,20,30,40,,)
        \connn(210,330)
        \connn(320,20)
        \connn(30,200)
	\end{tikzpicture}
    \centering
    \caption{Two Type 1 semi-triangles, $\mathcal{A}_C$ case}
    \label{fig:semicont2a}
  \end{subfigure}
  \quad
  \begin{subfigure}{.3\textwidth}
    \begin{tikzpicture}[scale=1]
        \cyclee
        \barcc(250,260,270,280,$A$,)
        \barcc(300,310,320,330,$C$,)
        \barcc(140,150,160,170,$B$,)
        \connn(270,320)
        \connn(310,150)
        \connn(160,260)
        \barcc(190,200,210,220,$D$,)
        \barcc(0,10,20,30,,)
        \barcc(50,60,70,80,,)
        \connn(210,20)
        \connn(10,60)
        \connn(70,200)
    \end{tikzpicture}
    \centering
    \caption{Two Type 1 semi-triangles, $\mathcal{A}_B$ case}
    \label{fig:semicont2b}
  \end{subfigure}
  \centering
  \caption{Contradicting cycles implied by semi-triangles}
\end{figure}

\section*{Acknowledgement}

I would like to thank my supervisors Peter Allen, Julia Böttcher and Jozef Skokan for recommending the problem, for discussions on the topic and for proofreading the paper.

\bibliographystyle{amsalpha}
\bibliography{biblio}

\providecommand{\bysame}{\leavevmode\hbox to3em{\hrulefill}\thinspace}
\providecommand{\MR}{\relax\ifhmode\unskip\space\fi MR }
\providecommand{\MRhref}[2]{%
  \href{http://www.ams.org/mathscinet-getitem?mr=#1}{#2}
}
\providecommand{\href}[2]{#2}
\begin{thebibliography}{Bon75}

\bibitem[Bon75]{bondyMeta}
J.~A. Bondy, \emph{Pancyclic graphs: recent results}, 181--187. Colloq. Math.
  Soc. J\'anos Bolyai, Vol. 10. \MR{0373957}

\bibitem[CE72]{chvatalErdos}
V.~Chv\'atal and P.~Erdős, \emph{A note on {H}amiltonian circuits}, Discrete
  Math. \textbf{2} (1972), 111--113. \MR{0297600}

\bibitem[Dir52]{dirac}
G.~A. Dirac, \emph{Some theorems on abstract graphs}, Proc. London Math. Soc.
  (3) \textbf{2} (1952), 69--81. \MR{0047308}

\bibitem[Erd74]{erdosConjecture}
P.~Erdős, \emph{Some problems in graph theory}, 187--190. Lecture Notes in
  Math., Vol. 411. \MR{0379287}

\bibitem[KS10]{KeevashSudakov}
P.~Keevash and B.~Sudakov, \emph{Pancyclicity of {H}amiltonian and highly
  connected graphs}, J. Combin. Theory Ser. B \textbf{100} (2010), no.~5,
  456--467. \MR{2644233}

\bibitem[LS12]{LeeSudakov}
C.~Lee and B.~Sudakov, \emph{Hamiltonicity, independence number, and
  pancyclicity}, European J. Combin. \textbf{33} (2012), no.~4, 449--457.
  \MR{2864429}

\end{thebibliography}

\end{document}